\documentclass[a4]{amsart}

\usepackage{amssymb}
\usepackage[all]{xy}
\usepackage{amsmath, amsfonts, amsthm , amssymb, amscd}
\usepackage{booktabs}
\newtheorem*{theorem*}{Theorem A}
\newtheorem{theorem}{Theorem}[section]

\newtheorem{proposition}[theorem]{Proposition}
\newtheorem{lemma}[theorem]{Lemma}
\newtheorem{corollary}[theorem]{Corollary}

\theoremstyle{definition}
\newtheorem{definition}[theorem]{Definition}

\theoremstyle{remark}

\newcommand{\Hom}{\ensuremath{\mathrm{Hom}}}
\newcommand{\sk}{\ensuremath{\mathrm{sk}}}
\newcommand{\lcm}{\ensuremath{\mathrm{lcm}}}
\newcommand{\Id}{\ensuremath{\mathrm{Id}}}
\newcommand{\rank}{\mathop{\mathrm{rank}}}
\newcommand{\N}{\ensuremath{\mathbb{N}}}
\newcommand{\Z}{\ensuremath{\mathbb{Z}}}

\newcommand{\T}{\ensuremath{\mathcal{T}}}
\newcommand{\J}{\ensuremath{\mathcal{J}}}

\newcommand{\lra}{\longrightarrow}
\newcommand{\dr}[3]{\ensuremath{#1\stackrel{#2}
{\longrightarrow}#3}}
\newcommand{\ddr}[5]{\ensuremath{#1\stackrel{#2}
{\longrightarrow}#3\stackrel{#4}{\longrightarrow}#5}}

\newcommand{\ddddr}[9]{\ensuremath{#1\stackrel{#2}
{\longrightarrow}#3\stackrel{#4}{\longrightarrow}#5
\stackrel{#6}{\longrightarrow}#7}\stackrel{#8}{\longrightarrow}#9}

\begin{document}
\title{The degrees of maps between $(2n-1)$-Poincar\' e complexes}
\author{Jelena Grbi\'{c}, Aleksandar Vu\v ci\' c}
\address{School of Mathematics, University of Southampton,
         Southampton SO17 1BJ, United Kingdom}
\email{J.Grbic@soton.ac.uk}

\address{Faculty of Mathematics, Belgrade University, Belgrade, Serbia}
\email{avucic@eunet.rs}


\begin{abstract}
In this paper, using exclusively homotopy theoretical methods,  we study degrees of maps between $(n-2)$-connected $(2n-1)$-dimensional Poincar\' e complexes which have torsion free integral homology. Necessary and sufficient algebraic conditions for the existence of map degrees between such Poincar\' e complexes are established.

We calculate the set of all map degrees between certain two $(n-2)$-connected $(2n-1)$-dimensional torsion free Poincar\'e complexes.

For low $n$, using knowledge of possible degrees of self maps, we classify, up to homotopy, torsion free $(n-2)$-connected $(2n-1)$-dimensional Poincar\' e complexes.
\end{abstract}

\maketitle


\section{Introduction}
One of the earliest homotopy invariants for a map $f\colon M\lra N$ between two closed oriented manifolds of the same dimension is its degree. The notion of degree originated from the main idea in Gauss's proofs for the fundamental theorem of algebra and was first formally formulated by Brouwer~\cite{Brouwer} for a map between spheres. Brouwer showed that the degree is a homotopy invariant, and used it to prove the Brouwer fixed point theorem. Brouwer's notion of degree in the years 1910-1912 preceded the rigorous development of homology, and used the technique of simplicial approximation. In analysis, the degree was at first defined  for smooth maps $f$ between manifolds of the same dimension as an algebraic account of the number of preimages of a regular point of $f$, then extended to continuous ones by using homotopy approximation by Milnor~\cite{Milnor}.

In modern mathematics, the degree of a map plays an important role in topology and geometry, and it serves as one of main tools in establishing existence for highly nonlinear problems~\cite{Siegberg}. In physics, the degree of a continuous map is one example of a topological quantum number.

The definition of degree in terms of homology is probably the most elegant one. Consider a map $f\colon M\lra N$ between closed oriented $n$-dimensional manifolds and the induced map in homology $f_*\colon H_*(M)\lra H_*(N)$. Denote by $[M]$ and $[N]$  the fundamental classes of $M$ and $N$, respectively, specified by the orientation.  Then the {\it  degree} $\deg(f)$ of the map $f$ is defined to be a unique integer such that $f_*[M]=\deg(f) [N]$.

Our approach in studying the degree of a map is homotopy theoretical in  nature and provides a further generalisation to continuous maps between a wider class of topological spaces. For a $CW$-complex $X$, let $\sk_n(X)$  denote the $n$th skeleton of $X$. Using an elementary homotopy theoretical approach, we can see that for two $n$-dimensional manifolds $M$ and $N$,  the degree of $f\colon M\lra N$ is $d$ if  and only if there is a commutative diagram of cofibrations
\[
\xymatrix{
S^{n-1}\ar[r]^-{\alpha}\ar[d]^{[d]} & \sk_{n-1}(M)\ar[r]\ar[d]^{f_|} & M\ar[d]^{f}\\
S^{n-1}\ar[r]^-{\beta} & \sk_{n-1}(N)\ar[r] & N}
\]
where $[d]\colon S^{n-1}\lra S^{n-1}$ denotes the $d$-degree map on the $(n-1)$-sphere and $f_|$ is the restriction of $f$ to $\sk_{n-1}(M)$.

From the above condition, it is readily seen that degree is not a property shared only by manifolds. For a finite $n\in\N$, we define degree of a map between any two $n$-dimensional $CW$-complexes which both have only one $n$-dimensional cell. In particular, we can talk about degree of a map between two connected Poincar\' e complexes of the same dimension.

Denote by $\T^n$ all $n$-dimensional $CW$-complexes with one $n$-cell and a fixed choice of a generator in $H_n(X, \sk_{n-1}(X))$.

\begin{definition}
\label{defdegree}
Let $f\colon X\lra Y$ be a map between $\T^n$-spaces. The {\it degree} $\deg(f)$ of the map $f$ is $d$ if and only if there is a commutative diagram of cofibrations 
\begin{equation}
\xymatrix{
S^{n-1}\ar[r]^-{\alpha}\ar[d]^{[d]} & \sk_{n-1}(X)\ar[r]\ar[d]^{f_|} & X\ar[d]^{f}\\
S^{n-1}\ar[r]^-{\beta} & \sk_{n-1}(Y)\ar[r] & Y}
\end{equation}
where $f_|$ is the restriction of $f$ to $\sk_{n-1}(X)$.
\end{definition}

This definition can be reformulated in homology terms by saying that a map $f\colon X\lra Y$ between $\T^n$-spaces has degree $d$ if the induced map in homology
\[
f_*\colon {H}_n((X, X\setminus\{pt\});\Z) \lra {H}_n((Y, Y\setminus\{pt\});\Z)
\]
satisfies that $f_*([X])=d[Y]$, where $[X]$ and $[Y]$ are a fixed choice of generators of ${H}_n((X, X\setminus\{pt\});\Z)$ and ${H}_n((Y, Y\setminus\{pt\});\Z)$, respectively.

\begin{definition}
For any two $CW$-complexes $X,Y\in\T^n$, define
\[
D(X,Y):=\{ \deg (f)\ |\ f\colon X\lra Y\}
\]
to be the set of all possible degrees of maps from $X$ to $Y$.
\end{definition}
A general question is to determine the set $D(X,Y)$ for given spaces $X,Y\in\T^n$. As mentioned earlier some work has been done on determining the set $D(M,N)$, where $M$ and $N$ are manifolds. In dimension 2 the answer is known~\cite{Edmonds} but in general the question is difficult. Recently, many achievements have been made for certain classes of 3-manifolds (see for example~\cite{Hayat, Rong, Soma, Wang} and therein references), but in dimensions higher than 3 there are few relevant works (see for example~\cite{Baralic, Baues, Duan, Hoffman, Mcgibbon}) to the best of our knowledge. A significant breakthrough in the subject was achieved by Duan and Wang~\cite{DuanWang1, DuanWang2} who gave necessary and sufficient algebraic conditions for the existence of a map degree  between two given closed $(n-1)$-connected $2n$-dimensional manifolds. Their algebraic conditions are obtained using geometry and topology of this wide class of manifolds.

Our approach in studying the degree of maps is novel in two senses: first, we study the degree of maps not only between manifolds but more generally between Poincar\' e complexes and second, we consider maps between certain family of odd dimensional (manifolds) Poincar\' e complexes. The methods used are predominantly homotopy theoretical and are centred around the Hilton-Milnor theorem and detailed analysis of how right multiplication distributes through addition (recall that the right distributivity law does not hold amongst continuous maps). Geometric information appearing in terms of intersection forms and linking numbers of manifolds, which are commonly used in the study of map degrees, we interpret in an analogous information about Whitehead products.

The main result gives necessary and sufficient algebraic conditions for the existence of a map degree between  $(n-2)$-connected Poincar\' e $(2n-1)$-complexes. The family of all  $(n-2)$-connected Poincar\' e $(2n-1)$-complexes $X$ with the finitely generated, torsion free homology group $H_{n-1}(X;\Z)$ will be denoted by $\J_n$. Other notations, needed to state the theorem, are explained in Sections 2 and 3.

\begin{theorem*}

Let $X,Y\in\J_n$ and $\geq 3$. Then there is a map $f\colon X\lra Y$ of degree $d$ if and only if the system of equations
\begin{equation}
\label {1.jednacina}
\sum^k_{i=1}a_{it}(p_i\alpha)+\sum^{2k}_{i=k+1}a_{it}(\eta p_i\alpha)+\left(\sum^k_{i=1}\binom{a_{it}}{2}h_{ii}^\alpha+\sum_{1\leq i<j\leq k}m^\alpha_{ij}a_{it}a_{jt}+\sum^k_{i=1}a_{it}a_{k+i t}\right)[\Id_{S^{n-1}},\Id_{S^{n-1}}]\eta\end{equation} \[=d(q_t\beta)\quad \text{ for }1\leq t\leq m 
\]
\begin{equation}
\sum^{2k}_{i=k+1}a_{it}(p_i\alpha)=d(q_t\beta) \quad \text{ for } m+1\leq t\leq 2m 
\end{equation}
\begin{equation}
\sum^k_{i=1}a_{is}a_{it}h_{ii}^\alpha+\sum_{1\leq i<j\leq k}(a_{is}a_{jt}+a_{it}a_{js})m_{ij}^\alpha+\sum_{i=1}^k(a_{is}a_{i+k t}+a_{it}a_{i+k s})\equiv dm^\beta_{st}\quad   (\hspace{-3.5mm}\mod 2)
\end{equation}
for $1\leq s<t\leq m$ 
\begin{equation}
\label{4.jednacina}
\sum_{i=1}^k a_{is}a_{i+k t}=d\delta_{s, t-m} \quad\text{ for } 1\leq s\leq m, m+1\leq t\leq 2m
\end{equation}
has an integral solution $\{ a_{is} \ | \ 1\leq i\leq 2k, 1\leq s\leq 2m\}$.
\end{theorem*}

Although we are illustrating our method on maps for determining degrees of maps between $(n-2)$-connected Poincar\' e $(2n-1)$-complexes, that is not its limitation. Namely, if applied to $(n-1)$-connected $2n$-dimensional manifolds, it fully recovers the results of Duan and Wang~\cite{DuanWang1, DuanWang2}.

In Section~\ref{examples} we apply our criterium to explicitly calculate the set $D(X,Y)$ for certain choices of $X,Y\in \J_n$.

Many research has been done in attempt to classify manifolds of any type. In Section 7, we concentrate on classifying, up to homotopy equivalences, torsion free $(n-2)$-connected $2n$-dimensional Poincar\' e complexes for $n\leq 7$. To do so, we use the knowledge of possible degrees of self maps. A curtail technical result is a characterisation of a homotopy equivalence as a map of degree $\pm 1$.

\section{Poincar\' e complexes}
\subsection{Poincar\' e complexes: definitions and main properties} In this paper we shall consider only simply-connected $CW$-complexes and thus adapt the definition of Poincar\' e complexes to the spaces at hand. Therefore, we say that a simply-connected $CW$-complex $X$ is called a \emph{Poincar\' e complex of formal dimension $n$}  (Poincar\' e $n$-complex for short) if  the Poincar\' e duality condition is satisfied, that is, if there is a class $[X]\in H_n(X; \Z)$ such that the cap product with it induces an isomorphism
\[
\cap [X]\colon \dr{H^*(X;\Z)}{\cong}{H_{n-*}(X; \Z)}.
\]

Wall~\cite{WallP} proved that a simply-connected $CW$-complex $X$ is a Poincar\' e $n$-complex if and only if the cap product map
\[
\cap [X]\colon \dr{H^*(X;M)}{}{H_{n-*}(X; M)}
\]
is an isomorphism for all $\Z$-modules $M$.

The Poincar\' e duality condition imposes certain restrictions to a possible cup product structure on $H^*(X)$.  Let $R$ be a field. The cup product pairing
 \[
 \dr{H^{k}(X; R)\otimes H^{n-k}(X; R)}{\cup}{H^{n}(X; R)\cong R}
 \]
is non-singular for each $k$. Equivalently, the maps
\[
H^{n-k}(X;R)\lra \Hom (H^k(X; R), R)
\]
\[
H^{k}(X;R)\lra \Hom (H^{n-k}(X; R), R)
\]
induced by the cup pairing are isomorphisms.

Over $\Z$, the induced paring
\[
 \dr{H^{k}(X; \Z)/\mathrm{Tors}\otimes H^{n-k}(X;\Z)/\mathrm{Tors}}{\cup}{\Z}
 \]
is non-singular for each $k$, where $\mathrm{Tors}$ denotes the torsion part of the appropriate abelian group.

\subsection{Torsion free $(n-2)$-connected Poincar\' e $(2n-1)$-complexes}

Let $\J_n$ denote the family of $(n-2)$-connected Poincar\' e $(2n-1)$-complexes $X$ with the finitely generated, torsion free homology group $H_{n-1}(X;\Z)$. If $\rank  H_{n-1}(X;\Z)=k$, we shall say that the Poincar\' e complex $X$ is of rank $k$. For $X\in\J_n$, let $\bar X$ denote the $(2n-2)$-skeleton of $X$, that is, $\bar X=\sk_{2n-2}(X)$.

Let $k=\rank H_{n-1}(X;\Z)$. Considering the space $ \bigvee_{i=1}^k (S^{n-1}_i\vee S^n_{i+k})$, fix a choice of its cohomology generators  $e^{n-1}_i$ and $e^n_{i}$ corresponding, under the inclusion, to generators of $H^{n-1}(S^{n-1}_i;\Z)$ and $H^n(S^n_{i+k})$, and by $e^{2n-1}$ a generator of $H^{2n-1}(X;\Z)$.

\begin{lemma}
\label{heq}
\begin{enumerate}
\item For $X\in\J_n$ of rank $k$, there is a homotopy equivalence
\[
h\colon \bigvee_{i=1}^k (S^{n-1}_i\vee S^n_{i+k})\lra \bar X.
\]
\item The equivalence $h$ can be chosen so that
\[
{h^*}^{-1}(e^{n-1}_i) \cup {h^*}^{-1}(e_j^n)=\delta_{ij}e^{2n-1} \text{ for }1\leq i,j\leq k
\]
that is, so that the cup product matrix of $X$ is the identity matrix $I_k$.
\end{enumerate}
\end{lemma}

\begin{proof}
\begin{enumerate}
\item Since $X\in\J_n$, the first non-trivial homology group is $H_{n-1}(X;\Z)\cong H_n(X;\Z)\cong\Z^k$. Using the Hurewicz isomorphism $h\colon \pi_{n-1}(X)\lra H_{n-1}(X;\Z)$, we see that the $(n-1)$-skeleton $\sk_{n-1}(X)$ is a wedge of $k$ spheres $S^{n-1}$. To form the $n$-skeleton $\sk_n(X)$,  some $n$-dimensional cells need to be attached by $f\colon S^{n-1}\lra \bigvee_{i=1}^kS_i^{n-1}$. The homotopy type of $f$ is determined by its degree $\deg(f)\in\Z$. Any non-trivial attaching map $f$ would produce a $\deg(f)$-torsion class in homology. Therefore all the attaching maps are null-homotopic, by Poincar\' e duality $k$ $n$-dimensional cells need to be attached, and thus $\bar X\simeq \bigvee_{i=1}^k (S^{n-1}_i\vee S^n_{i+k})$.

\item We shall start by fixing homology and cohomology generators of $X$.
First, choose $(n-1)$-dimensional generators $e^{n-1}_1,\ldots, e^{n-1}_k$ in the cohomology  group $H^{n-1}(X)$ to be dual to the homology basis  $\{ x_t \ |\ t=1,\ldots, k\}$ of $H_{n-1}(X;\Z)$ obtained under the inclusions $\iota_i\colon S^{n-1}_{i}\lra X$. Using Poincar\' e duality, choose a basis $\{ e^n_t\ |\ t=1,\ldots ,k\}$ of $H^n(X;\Z)$ such that $x_t=e^n_t\cap e$, where $e$ is a fixed generator of $H_{2n-1}(X;\Z)$. Dualise the basis of $H^{n}(X)$ to obtain the dual basis $\{ y_t \ |\ t=1,\ldots, k\}$ of $H_{n}(X;\Z)$, that is, $\langle y_i, e^{n}_j\rangle=\delta_{ij}\equiv e^{n-1}_j(y_i)$ for $1\leq i,j\leq k$.

Denote by $(a_{ij})$ the cup product matrix of $X$, that is, $e^{n-1}_i\cup e^n_j=a_{ij}e^{2n-1}$ for $1\leq i,j\leq k$. The way the cohomology generators are chosen simplifies the cup product matrix. Namely,
\[
a_{ij}=(a_{ij}e^{2n-1})\cap e=(e^{n-1}_i\cup e^{n}_j)\cap e=e^{n-1}_i\cap (e^n_j\cap e)=e^{n-1}_i\cap x_j=\delta_{ij}
\]
for $1\leq i,j\leq k$.
In other words, since $x_i$ and $y_i$ are spherical classes, there are maps $\iota_i\colon S^{n-1}\lra \bar X$ and $j_i\colon S^n\lra \bar X$ such that $h=\bigvee_i (\iota_i\vee j_i)$ is a homotopy equivalence. In this way we choose a cohomology basis of $X$ for which the cup product matrix is the identity matrix $I_k$.

\end{enumerate}
\end{proof}

\section{Homotopy invariants of attaching maps}

For torsion free $(n-2)$-connected Poincar\' e  $(2n-1)$-complexes $X$ and $Y$, denote by $\alpha\colon S^{2n-2}\lra \bar X$ and $\beta\colon S^{2n-2}\lra\bar Y$ the corresponding attaching maps of the  top cell.
Definition~\ref{defdegree} applied to $X,Y\in\J_n$ gives the following homotopy theoretical condition  for a map $f\colon X\lra Y$ to have degree $d$.

\begin{lemma}
\label{lemmadegree}
Let $X,Y\in\J_n$ and $f\colon X\lra Y$ a continuous map. The {\emph degree} $\deg(f)=d$ if and only if there is a commutative diagram
\begin{equation}
\label{stependiagram2}
\xymatrix{
S^{2n-2}\ar[r]^-{\alpha}\ar[d]^{[d]} & \bar X\ar[d]^{f_|} \\
S^{2n-2}\ar[r]^-{\beta} & \bar Y.}
\end{equation}
where $f_|$ is the restriction of $f$ on $\bar X$.
\qed
\end{lemma}

The lemma suggests that in order to detect the degree of the map $f$ a better understanding of the homotopy invariants of the attaching maps $\alpha$ and $\beta$ is needed. Our main tool for describing these homotopy invariants is the Hilton-Milnor theorem (see, for example,~\cite[Chapter XI, pg 511]{Whitehead}).

\subsection{Hilton-Milnor Theorem}
For topological spaces $X$ and $Y$, let $\iota_X\colon X\lra X\vee Y$ and $\iota_Y\colon Y\lra X\vee Y$ be the inclusions of the left and right summands. By abuse of notation, we also identify $\dr{\iota_X=E\circ\iota_X\colon X\lra X\vee Y}{E}{\Omega\Sigma (X\vee Y)}$ and $\dr{\iota_Y=E\circ\iota_Y\colon Y\lra X\vee Y}{E}{\Omega\Sigma (X\vee Y)}$, where $E\colon Z\lra\Omega\Sigma Z$ is the suspension map, which is adjoint of the identity $\Sigma Z\lra\Sigma Z$.

The {\it Samelson product} of $\iota_X$ and $\iota_Y$ is the map $[\iota_X,\iota_Y]\colon X\wedge Y\lra\Omega\Sigma(X\vee Y)$ defined by the composite
\[
\ddr{X\wedge Y}{\iota_X\wedge \iota_Y}{\Omega\Sigma(X\vee Y)\wedge\Omega\Sigma(X\vee Y)}{[\ ,\ ]}{\Omega\Sigma(X\vee Y)}
\]
where $[\alpha,\beta]=\alpha\beta\alpha^{-1}\beta^{-1}$.

Write $ad(\alpha)(\beta)=[\alpha, \beta]$. We can form iterated Samelson products
\[
ad(\iota_X)^i(\iota_Y )\colon X^i\wedge Y\lra\Omega\Sigma (X\vee Y)
\]
and add them up to get a map of the infinite wedge
\[
\bigvee_{i\geq 0} ad(\iota_X)^i(\iota_Y )\colon\bigvee_{i\geq 0}X^i\wedge Y\lra\Omega\Sigma(X\vee Y).
\]
Now using the universal property of the James construction~\cite{James}, form the multiplicative extensions
\[
\overline{\iota_X}=\Omega\Sigma(\iota_X)\colon\Omega\Sigma X\lra\Omega\Sigma(X\vee Y)
\]
and
\[
\overline{\bigvee_{i\geq 0} ad(\iota_X)^i(\iota_Y )}\colon\Omega\Sigma(\bigvee_{i\geq 0}X^i\wedge Y)\lra\Omega\Sigma(X\vee Y).
\]
Finally use the multiplication of $\Omega\Sigma(X\vee Y)$ to multiply these maps and to get the {\it Hilton-Milnor theorem} which asserts the following. If $X$ and $Y$ are connected, then there is a weak homotopy equivalence
\[
\Theta\colon\Omega\Sigma X\times\Omega\Sigma(\bigvee_{i\geq0}X^i\wedge Y)\lra\Omega\Sigma(X\vee Y).
\]
Of course, if $X$ and $Y$ are both CW-complexes then the above map is a homotopy equivalence.

By induction, the Hilton-Milnor theorem naturally generalises to a decomposition of the loop space on the wedge sum of $n$ suspension spaces
\[
\Omega\Sigma (\bigvee_{i=1}^n X_i)\simeq \prod_{\omega\in B}\Omega\Sigma(\omega(X_1,\ldots, X_n))\simeq\prod_{1\leq i\leq n}\Omega\Sigma X_i\times\prod_{1\leq i<j\leq n}\Omega\Sigma (X_i\wedge X_j)\times\cdots
\]
where $B$ is a Hall basis generated by the set $L=\{ \iota_{X_1},\ldots, \iota_{X_n}\}$ (see for example~\cite[Chapter XI]{Whitehead}).

For future reference, let us recall that the {\it $\omega$-Hopf-Hilton} invariant (see for example~\cite[Chapter XI]{Whitehead})
\[
H_{\omega}\colon\Omega\Sigma (\bigvee_{k=1}^n X_k)\lra \Omega\Sigma(\omega(X_1,\ldots, X_n))
\]
is defined by
\[
H_\omega= q_\omega\circ\Theta^{-1}
\]
where $q_\omega\colon  \prod_{\omega\in B}\Omega\Sigma(\omega(X_1,\ldots, X_n))\lra \Omega\Sigma(\omega(X_1,\ldots, X_n))$
is the projection onto the $\omega$ factor.

\subsection{Decomposing maps using the Hilton-Milnor theorem}

As a corollary of the Hilton-Milnor theorem, we have the following result.

\begin{proposition}
\label{HM}
Let $X$ and $Y$ be spaces such that $Y\simeq\Sigma\bigvee_{i=1}^k Y_i$. Let $n=\dim \Sigma X$ and all $\Sigma Y_i$ be $r$-connected.
Then any map $f\colon\Sigma X\lra \Sigma\bigvee_{i=1}^k Y_i$ can be uniquely decomposed in the following way.
\begin{itemize}
\item[\text{$[I]$}] If $n\leq 2r$, then
\[
f=\sum_{i=1}^k\iota_ip_if
\]
where $p_i\colon Y\lra \Sigma Y_i$ is the pinch map and $\iota_i\colon \Sigma Y_i\lra Y$ is the inclusion.
\item [\text{$[II]$}] If $n\leq 3r$, then
\[
f=\sum_{i=1}^k\iota_ip_if +\sum_{1\leq i<j\leq k}[\iota_i, \iota_j]M_{ij}(f)
\]
where $M_{ij}(f)\colon \Sigma X\lra\Sigma Y_i\wedge Y_j$ is the $(i,j)$-Hilton-Milnor invariant of~$f$ defined by the composite 
\[
\ddddr{\Sigma X}{f}{\Sigma \bigvee_{i=1}^k Y_i}{\Sigma E}{\Sigma\Omega\Sigma\bigvee_{i=1}^k Y_i}{\Sigma H_{ij}}{\Sigma\Omega\Sigma Y_i\wedge Y_j}{\mathrm{ev}}{\Sigma Y_i\wedge Y_j}.
\]
\qed
\end{itemize}
\end{proposition}
It is easy to see that $M_{ij}\colon [\Sigma X, \Sigma\bigvee_{i=1}^k Y_i ]\lra[\Sigma X,\Sigma Y_i\wedge Y_j]$ is a homomorphism.

Prompted by the decomposition of the map $f$ in the previous proposition, we shall introduce some notions which will be used throughout the paper.

For a map $f\colon X\lra Y=\bigvee_{i=1}^k Y_i$, the part of the homotopy invariants of $f$
\[
p_if\quad \text{ for } 1\leq i\leq k
\]
will be called the {\it first order homotopy invariants of} $f$, while  we shall refer to
\[
M_{ij}(f)\quad \text{ for } 1\leq i<j\leq k
\]
as the {\it second order homotopy invariants of} $f$.

If $f$ satisfies condition [I] of Proposition~\ref{HM}, that is, if $n\leq 2r$, we shall say that $f$ is of {\it type [I]}. If $f$ satisfies condition [II] of Proposition~\ref{HM}, that is, if $n\leq 3r$, we shall say that $f$ is of {\it type [II]}.
\begin{proposition}
\label{htpyinv}
Let $X\in\J_n$ with a top cell attaching map
\[
\alpha\colon S^{2n-2}\lra \bar X\simeq \bigvee_{i=1}^k {(S^{n-1}_i\vee S^n_{i+k})}.
\]
If $n>3$, then $\alpha$ decomposes in the following way
\[
\alpha=\sum^{2k}_{i=1}\iota_ip_i\alpha +\sum_{1\leq i<j\leq k}m^\alpha_{ij}[\iota_i, \iota_j]\eta +\sum^k_{i=1}[\iota_i, \iota_{i+k}]
\]
where $m^\alpha_{ij}$ takes values 0 or 1 and $\eta\colon S^{2n-2}\lra S^{2n-3}$ is the Hopf map.
\end{proposition}
\begin{proof}
Using Proposition~\ref{HM}, the attaching map $\alpha\colon S^{2n-2}\lra \bar X=\bigvee_{i=1}^k {(S^{n-1}_i\vee S^n_{i+k})}$, being of type [II], can be decomposed as
\begin{equation}
\label{alphaspheres}
\alpha=\sum^{2k}_{i=1}\iota_ip_i\alpha +\sum_{1\leq i<j\leq 2k}[\iota_i, \iota_j]M_{ij}(\alpha).
\end{equation}
The second sum can be further decomposed in the following way.

i) Recall that $\pi_{2n-2}(S^{2n-3})\cong\Z/2$ generated by $\eta$. Therefore for $1\leq i<j\leq k$, the map $M_{ij}(\alpha)\colon S^{2n-2}\lra\Sigma S^{n-2}\wedge S^{n-2}$ is $M_{ij}(\alpha)=m_{ij}^\alpha\eta$ where $m_{ij}^\alpha\in\Z/2$.

ii) For $1\leq i\leq k$ and  $k+1\leq j\leq 2k$, the map $M_{ij}(\alpha)\colon S^{2n-2}\lra\Sigma S^{n-2}\wedge S^{n-1}$ is $M_{ij}(\alpha)=m_{ij}^\alpha\Id_{S^{2n-2}}$ where $m_{ij}^\alpha$ is the cup product $a_{ij}$. By Lemma~\ref{heq}, the cup product $a_{ij}$ is identified by $\delta_{ij}$. To see this, consider the mapping cone $\tilde X=(S^{n-1}_i\vee S^n_{j})\cup_{[\iota_i,\iota_j]}e^{2n-1}=S^{n-1}\times S^{n}$ mapping to $X$. If, as before, $e^{n-1}_i\in H^{n-1}(\bar{X}; \Z)$ denote the generators corresponding to $S^{n-1}$ and $e^n_j\in H^{n}(\bar{X}; \Z)$ denote the generators corresponding to  $S^{n}$, then the cup product $e^{n-1}_i\cup e^n_j$ represents the generator of $H^{2n-1}(\tilde {X}; \Z)$. This shows that the coefficients $m_{ij}^\alpha$  are determined by the cup product matrix $A=(a_{ij})$ which is the identity matrix.

iii) For $k+1\leq i<j\leq 2k$, the map $M_{ij}(\alpha)\colon S^{2n-2}\lra\Sigma S^{n-1}\wedge S^{n-1}$ is trivial for connectivity reasons.

Putting all these three cases together, one proves the proposition.
\end{proof}

In this way, Proposition~\ref{htpyinv} shows that the complete set of homotopy invariants of $\alpha$ is given by
\[
\begin{array}{ll}
p_i\alpha\colon S^{2n-2}\lra S^{n-1} & \text{ for } 1\leq i\leq k\\
p_i\alpha\colon S^{2n-2}\lra S^{n} & \text{ for } k+1\leq i\leq 2k\\
m_{ij}^\alpha & \text{ for } 1\leq i<j\leq k\\
\end{array}
\]
where $m_{ij}^\alpha$ may take values $0$ or $1$.

Note that these invariants of $\alpha$ are independent and for a different choice of invariants the resulting attaching maps $\alpha$ are not homotopic.

\section{Decomposing maps between Poincar\' e complexes}

Our next aim is to study maps between the $(2n-2)$-skeletons of Poincar\' e complexes $X,Y\in\J_n$.
We start with a more general statement which is a direct consequence of the Hilton-Milnor theorem.
\begin{proposition}
\label{3.4}
Let $X=\Sigma\bigvee^k_{i=1} X_i$ be $d$-dimensional $CW$-complex and let $Y=\Sigma\bigvee^m_{s=1} Y_s$ be $r$-connected $CW$-complex.  Let $P\colon \Sigma\bigvee^k_{i=1} X_i\lra\Sigma\bigvee^m_{s=1} Y_s$. If $d\leq 2r$, then
\[
P=\sum^k_{i=1}\sum^m_{s=1}j_s P_{is}p_i
\]
where $P_{is}=q_sP\iota_i\colon\Sigma X_i\lra \Sigma Y_s$, $\iota_i$ and $j_s$ are the corresponding inclusions, while $p_i$ and $q_s$ are corresponding pinch maps.
\end{proposition}
\begin{proof}
A proof is obtained as a direct application of Proposition~\ref{HM}, noting that the map $P$ is of type [I].
\end{proof}

\subsection{The Hopf-Hilton invariants}

For $d$-dimensional $CW$-complex $\Sigma U$ and $r$-connected $CW$-complex $\Sigma W$ such that $d\leq 3r$, consider a map $\alpha\colon \Sigma U\lra\Sigma W$ of type [II] and denote by $i^1, i^2\colon \Sigma W\lra\Sigma W\vee\Sigma W$ the inclusions to the left and right summands. Then their sum $i^1+i^2\colon\Sigma W\lra \Sigma W\vee\Sigma W$ is the standard comultiplication induced by the suspension. The {\it Hopf-Hilton invariant of~$\alpha$} can be defined as
\[
H(\alpha)=M_{12}((i^1+i^2)\alpha)\colon\Sigma U\lra \Sigma W\wedge W
\]
where $M_{12}(f)$ denotes the $(1,2)$-Hilton-Milnor invariant of $f$.

We list some of the properties of $H(\alpha)$ which are generalisations to those proved by Whitehead~\cite[Chapter XI, Section 8]{Whitehead} when $U$ and $W$ are spheres. It is worth noticing that $H(\alpha)$ is an obstruction to the right distributivity law to hold for the map $\alpha$.

\begin{proposition}
Let $\alpha\colon \Sigma U\lra \Sigma W$ be a map of type [II], where $d=\dim \Sigma U$ and $\Sigma W$ is $r$-connected.
\begin{enumerate}
\item If $d\leq 2r$, then $H(\alpha)=0$.
\item If $d\leq 3r$, then $H(H(\alpha))=0$.

\item For any $\beta_1,\beta_2\colon \Sigma W\lra\Sigma Z$, it holds
\[
(\beta_1+\beta_2)\alpha=\beta_1\alpha+\beta_2\alpha+[\beta_1,\beta_2]H(\alpha).
\]
\item If $\beta_i\colon \Sigma W\lra\Sigma Z$ for $1\leq i\leq n$, then
\[
\left( \sum_{i=1}^n\beta_i\right)\alpha=\sum_{i=1}^n\beta_i\alpha+\sum_{1\leq i<j\leq n}[\beta_i, \beta_j]H(\alpha).
\]
\item
For $\beta\colon \Sigma W\lra\Sigma Z$ and $k\in\Z$, it holds
\[
(k\beta)\alpha=k(\beta\alpha)+\binom{k}{2}[\beta,\beta]H(\alpha).
\]
\item Considered as a map
\[
H\colon [\Sigma U,\Sigma W]\lra [\Sigma U,\Sigma W\wedge W]
\]
$H$ is a homomorphism.
\item For $k\in \Z$, $H(\alpha [k])=kH(\alpha)$.
\item For $k\in \Z$, $H([k]\alpha)=k^2 H(\alpha)$.
\end{enumerate}
\end{proposition}

\begin{proof}
\begin{enumerate}
\item For dimensional reasons, since $\dim \Sigma U$ is less or equal to the connectivity of $\Sigma W\wedge W$, we have that $H(\alpha)\colon \Sigma U\lra \Sigma W\wedge W$ is trivial.
\item If $\alpha$ is of type [II], then $H(\alpha)$ is of type [I] and the statement follows by property (1).
\item $(\beta_1+\beta_2)\alpha=\nabla (\beta_1\vee \beta_2)(i^1+i^2)\alpha=\nabla(\beta_1\vee \beta_2)(i^1\alpha+i^2\alpha+[i^1,i^2]H(\alpha))=\beta_1\alpha+\beta_2\alpha+[\beta_1,\beta_2]H(\alpha)$, where $\nabla$ is the fold map.
\item  The proof is by induction on the number of summands. Property (3) establishes the base of mathematical induction. Noting that $H(\alpha)$ distributes from the right, as a consequence of properties (2) and (3), we have
\[
\left( \sum_{i=1}^{n-1}\beta_i+\beta_n\right)\alpha=(\sum_{i=1}^{n-1}\beta_i)\alpha+\beta_n\alpha+[\sum_1^{n-1}\beta_i, \beta_n]H(\alpha)
\]
\[
= \sum_{i=1}^n\beta_i\alpha+\sum_{1\leq i<j\leq n-1}[\beta_i, \beta_j]H(\alpha)+\sum_{i=1}^{n-1}[\beta_i,\beta_n]H(\alpha)= \sum_{i=1}^n\beta_i\alpha+\sum_{1\leq i<j\leq n}[\beta_i, \beta_j]H(\alpha)
\]
\item For positive $k$, the statement is a direct corollary of property (4). Applying properties (2) and (3) to $H(\alpha)$, we have
\[
(\beta_1+\beta_2)H(\alpha)=\beta_1 H(\alpha) +\beta_2 H(\alpha)
\]
and therefore
\[
(-\beta)H(\alpha)=-(\beta H(\alpha)).
\]
Now
\[
0=(\beta -\beta)\alpha=\beta \alpha + (-\beta)\alpha + [\beta , -\beta]H(\alpha)
\]
and therefore
\[
(-\beta)\alpha = -(\beta \alpha) + [\beta , \beta]H(\alpha )
\]
which proves the statement for $k=-1$.

For $k\geq 2$, we get
\[
(-k\beta)\alpha=-((k\beta)\alpha)+[k\beta,k\beta]H(\alpha)=-(k(\beta\alpha)+\binom{k}{2}[\beta,\beta]H(\alpha))+k^2[\beta,\beta]H(\alpha)
\]
\[
=-k(\beta\alpha)+\binom{-k}{2}[\beta, \beta] H(\alpha).
\]

\item For $\alpha,\beta\in [\Sigma U,\Sigma W]$, we have
\[
H(\alpha+\beta)=M_{12}((i^1+i^2)\alpha +(i^1+i^2)\beta)=M_{12}((i^1+i^2)\alpha)+M_{12}((i^1+i^2)\beta)=H(\alpha)+H(\beta).
\]
\item Since $\alpha[k]=k\alpha$, the property follows readily as $H$ is a homomorphism.
\item $[i^1,i^2]H([k]\alpha)=(ki^1+ki^2)\alpha-(ki^1)\alpha-(ki^2)\alpha=[ki^1,ki^2]H(\alpha)=(k^2[i^1,i^2])H(\alpha)=[i^1,i^2](k^2H(\alpha))$. Now using the uniqueness of the Hilton-Milnor decomposition, $H([k]\alpha)=k^2 H(\alpha)$ for $k\in \Z$.
\end{enumerate}
\end{proof}

Note that for $\alpha\colon S^{2r-1}\lra S^r$, we have
$H(\alpha)=H_0(\alpha)\Id_{S^{2r-1}}$ where $H_0(\alpha)\in\Z$ is the Hopf invariant of $\alpha$.

\subsection{The Hopf-Hilton invariant of a map with a wedge sum codomain}

For the remaining of the section, let us assume that $\alpha\colon S^{2n-2}\lra \Sigma \bigvee_{i=1}^{k}X_i$ is of type [II].
Then
\[
H(\alpha)\colon S^{2n-2}\lra \Sigma ( \bigvee_{i=1}^{k}X_i)\wedge ( \bigvee_{i=1}^{k}X_i)=\Sigma  \bigvee_{i,j=1}^{k}(X_i\wedge X_j)
\]
is of type [I].
Therefore, Proposition~\ref{3.4} implies that
\[
H(\alpha)=\sum^{k}_{i,j=1}\iota_{ij}H_{ij}(\alpha)
\]
where $H_{ij}(\alpha)=p_{ij}H(\alpha)\colon S^{2n-2}\lra\Sigma X_i\wedge X_j$, the map $p_{ij}\colon \Sigma\bigvee_{i,j=1}^{k}(X_i\wedge X_j) \lra \Sigma X_i\wedge X_j$ is the pinch map, and $\iota_{ij}\colon\Sigma X_i\wedge X_j\lra\Sigma  \bigvee_{i,j=1}^{k}(X_i\wedge X_j)$ is the inclusion.

\begin{proposition}
\label{Hij}
For $H_{ij}(\alpha)$ the following holds.
\begin{enumerate}
\item $H_{ij}(\alpha)=M_{ij}(\alpha)$ for $i<j$.
\item Let $X_i$ and $X_j$ be spheres. Then \[H_{ji}(\alpha)=(-1)^{\dim(\Sigma X_i)\dim(\Sigma X_j)}H_{ij}(\alpha)\] for $i<j$.
\item $H_{ii}(\alpha)=H(p_i\alpha)$, where $p_i\colon \Sigma\bigvee_{i=1}^k X_i\lra \Sigma X_i$ is the pinch map.
\end{enumerate}
\end{proposition}
\begin{proof}
\begin{enumerate}
\item Being of type [II], the map $\alpha\colon S^{2n-2}\lra \Sigma \bigvee_{i=1}^{2k}X_i$ can be decomposed in the following way
\[
\alpha=(\sum^{k}_{i=1}\iota_ip_i)\alpha=\sum_{i=1}^k\iota_ip_i\alpha+\sum_{1\leq i<j\leq k}[\iota_i,\iota_j]H_{ij}(\alpha).
\]
When this decomposition is compared with the Hilton-Milnor decomposition of $\alpha$, the statement follows.
\item Let $i<j$. Then
\[
(\iota_ip_i+\iota_jp_j)\alpha =\iota_ip_i\alpha+\iota_jp_j\alpha+[\iota_i,\iota_j]H_{ij}(\alpha)
\]
and
\[
(\iota_jp_j+\iota_ip_i)\alpha =\iota_jp_j\alpha+\iota_ip_i\alpha+[\iota_j,\iota_i]H_{ji}(\alpha).
\]
The statement follows by comparing these identities, using the graded commutativity of the Whitehead products and the uniqueness of the Hilton-Milnor decomposition.
\item
For any topological space $Z$ and maps $f,g \colon \Sigma X_i \lra \Sigma Z$, it holds that
\[
(fp_i + gp_i)\alpha = (f+g)p_i\alpha=
fp_i\alpha + gp_i\alpha + [f,g]H(p_i\alpha).
\]

On the other hand,
\[
(fp_i + gp_i)\alpha =fp_i\alpha + gp_i\alpha + [fp_i,gp_i]H(\alpha)=
fp_i\alpha + gp_i\alpha + [f,g]H_{ii}(\alpha)
\]
Specialising to $Z=X \vee X$, $f = i^1$ , $g= i^2$ and using the uniqueness of the Hilton-Milnor theorem, the statement follows.
\end{enumerate}
\end{proof}

We now specialise to the case when all $X_i$'s are spheres.
\begin{proposition}
For $\alpha\colon S^{2n-2}\lra \bar X=\Sigma \bigvee_{i=1}^k(S^{n-2}_i\vee S^{n-1}_{i+k})$ and $n>3$, the following holds.

\begin{enumerate}
\item$H_{ij}(\alpha)=H_{ji}(\alpha)=M_{ij}(\alpha) \text{ for } 1\leq i<j\leq 2k$.
\item $H_{ij}(\alpha)=0$ for $k+1\leq i,j\leq 2k$.
\end{enumerate}
\end{proposition}
\begin{proof}
\begin{enumerate}
\item The statement follows readily from Proposition~\ref{Hij} (1) and (2). Although the sign $(-1)^{\dim(\Sigma X_i)\dim(\Sigma X_j)}$ might in general appear for some $n$ when $\dim(\Sigma X_i)=\dim(\Sigma X_j)$ that does not happen for the following reasons. On the one hand, if $\dim(\Sigma X_i)=\dim(\Sigma X_j)=n-1$, then $H_{ij}(\alpha)\in\Z/2$, so the sign does not matter. On the other hand, if $\dim(\Sigma X_i)=\dim(\Sigma X_j)=n$, then $H_{ij}(\alpha)\colon S^{2n-2}\lra S^{2n-1}$ is trivial.
\item As $H_{ij}(\alpha)\colon S^{2n-2}\lra S^{2n-1}$ for connectivity reasons, $H_{ii}(\alpha)=0$.
\end{enumerate}
\end{proof}

\subsection{The set of homotopy invariants of $P\circ\alpha$ }
In this section we denote by $\bar X=\Sigma \bigvee_{i=1}^k (S^{n-2}_i\vee S^{n-1}_{i+k})$, $\bar Y=\Sigma \bigvee^m_{s=1}(S^{n-2}_s\vee S^{n-1}_{s+m})$ and a map between them by $P\colon \bar X\lra \bar Y$. As before, the attaching map of the top cell in $X$ is $\alpha\colon S^{2n-2}\lra\bar X$.

By Proposition~\ref{3.4}, we can decomposed $P$ as
\begin{equation}
\label{decomP}
P=\sum_{i= 1}^{ 2k}\sum_{s=1}^{2m}j_sP_{is}p_i
\end{equation}
where $P_{is}=q_sP\iota_i\colon \Sigma X_i\lra \Sigma Y_s$ for $1\leq i\leq 2k$, $1\leq s\leq 2m$, $\Sigma X_i=S^{n-1}_i$  for $1\leq i\leq k$ and $\Sigma X_i=S^n_{i-k}$ for $k+1\leq i\leq 2k$ and analogously for $\Sigma Y_s$.

Notice that the group $[\Sigma X_i, \Sigma Y_s]$  is cyclic and generated by
\[
\left\{\begin{array}{llll}
\Id_{S^{n-1}} & 1\leq i\leq k, & 1\leq  s\leq m & \text{of infinite order }\\
0 & 1\leq i\leq k, & m+1\leq s\leq 2m &\\
\Id_{S^n} & k+1\leq i\leq 2k, & m+1\leq s\leq 2m & \text {of infinite order }\\
\eta & k+1\leq i\leq 2k, & 1\leq s\leq m& \text{of order 2.}\\
\end{array}\right.
\]
If we denote the generator of $[\Sigma X_i, \Sigma Y_s]$ by $\xi_{is}$, then $P_{is}=a_{is}\xi_{is}$ for some $a_{is}\in \Z$. Note that $P_{is}\colon\Sigma X_i\lra \Sigma Y_s$ can be desuspended, that is, there is $\bar P_{is}\colon X_i\lra Y_s$ such that $P_{is}=\Sigma \bar P_{is}$. Denote by $\bar \xi_{is}$ the desuspension of $\xi_{is}$. Using this notation, the map $P$ can be rewritten as
\begin{equation}
\label{mapP}
P=\sum^{2k}_{i=1}\sum^{2m}_{s=1}j_sa_{is}\xi_{is} p_i.
\end{equation}

We shall calculate the set of invariants of $P\circ \alpha$ consisting of
\begin{enumerate}
\item the first order homotopy invariants of $P\circ\alpha$, that is, $q_tP\alpha\colon S^{2n-2}\lra\Sigma Y_t$ for $1\leq t\leq 2m$ and
\item the second order homotopy invariants of $P\circ\alpha$, that is,
$M_{st}(P\circ\alpha)$ for $1\leq s<t\leq 2m$.
\end{enumerate}

We start by identifying the first order homotopy invariants of $P\circ\alpha$. Fix $t\in\{1,\ldots, 2m\}$. Then
\[
q_tP\alpha=\left( \sum_{i=1}^{2k}P_{it}p_i\right)\alpha=\sum^{2k}_{i=1}P_{it}p_i\alpha+\sum_{1\leq i<j\leq 2k}[P_{it}p_i, P_{jt}p_j]H(\alpha)=
\]
\[
\sum^{2k}_{i=1}P_{it}p_i\alpha+\sum_{1\leq i<j\leq 2k}[\Id_{\Sigma Y_t},\Id_{\Sigma Y_t}]\Sigma(\bar {P}_{it}\wedge \bar{P}_{jt})\Sigma (\bar p_i\wedge\bar p_j)H(\alpha)=
\]
\[
\sum^{2k}_{i=1}P_{it}p_i\alpha+[\Id_{\Sigma Y_t},\Id_{\Sigma Y_t}]\left(\sum_{1\leq i<j\leq 2k}a_{it}a_{jt}\Sigma(\bar\xi_{it}\wedge\bar \xi_{jt})H_{ij}(\alpha)\right).
\]

To calculate $\Sigma(\bar{\xi}_{it}\wedge\bar{\xi}_{jt})$, we consider two cases.
\begin{itemize}
\item[(a)]
Let $1\leq t\leq m$ and therefore $\Sigma Y_t=S^{n-1}_t$. There are three cases.
\begin{enumerate}
\item For $1\leq i<j\leq k$, $\Sigma X_i=\Sigma X_j=S^{n-1}$ and therefore $\xi_{it}=\xi_{jt}=\Id_{S^{n-1}}$ and
\[
\Sigma(\bar \xi_{it}\wedge\bar \xi_{jt} )=\Sigma (\Id_{S^{n-2}}\wedge\Id_{S^{n-2}})=\Id_{S^{2n-3}}.
\]
\item For $1\leq i\leq k, k+1\leq j\leq 2k$, $\Sigma X_i=S^{n-1}, \Sigma X_j=S^n$ and therefore $\xi_{it}=\Id_{S^{n-1}}$ and $\xi_{jt}=\eta$. Thus we have
\[
\Sigma(\bar \xi_{it}\wedge\bar \xi_{jt} )=\eta.
\]
\item For $k+1\leq i< j\leq 2k$, $\Sigma X_i=\Sigma X_j=S^n$ and $\xi_{it}=\xi_{jt}=\eta$. Therefore
\[
\Sigma(\bar \xi_{it}\wedge\bar \xi_{jt} )=\Sigma(\eta\wedge\eta)\colon\Sigma S^{n-1}\wedge S^{n-1}\lra \Sigma S^{n-2}\wedge S^{n-2}.
\]
\end{enumerate}

\item[(b)] For $m+1\leq t\leq 2m$, $\Sigma Y_t=S^n$. There are two cases:
\begin{enumerate}
\item For $1\leq i\leq k$, $\Sigma X_i=S^{n-1}$ and therefore $\xi_{it}=0$, that is,
\[
\Sigma(\bar \xi_{it}\wedge\bar \xi_{jt} )=0.
\]
\item For $k+1\leq i<j\leq 2k$, $\Sigma X_i=\Sigma X_j=S^n$ and therefore $\xi_{it}=\xi_{jt}=\Id_{S^n}$ and
\[
\Sigma(\bar \xi_{it}\wedge\bar \xi_{jt} )=\Id_{S^{2n-1}}.
\]
\end{enumerate}
\end{itemize}

Using the fact that $H_{ij}(\alpha)=M_{ij}(\alpha)$ for $i<j$ and the calculation of $M_{ij}(\alpha)$ in the proof of Proposition~\ref{htpyinv} implies that
\[
\Sigma (\bar \xi_{it}\wedge \bar \xi_{jt})H_{ij}(\alpha)=\left\{\begin{array}{ll}
m_{ij}^\alpha\eta & \text{ for } 1\leq t\leq m, 1\leq i<j\leq k\\
\delta_{i,j-k}\eta & \text{ for } 1\leq t\leq m, 1\leq i\leq k, k+1\leq j\leq 2k\\
0 & \text{ otherwise.}\\
\end{array}\right.
\]
This analsys proves the following proposition.
\begin{proposition}
\label{1storder}
The first order homotopy invariants of $P\circ\alpha$ are given by
\[
q_tP\alpha=\sum^{2k}_{i=1}P_{it}p_i\alpha
+\left(\sum_{1\leq i<j\leq k}m^\alpha_{ij}a_{it}a_{jt}+\sum^k_{i=1}a_{it}a_{k+i t}\right)[\Id_{S^{n-1}},\Id_{S^{n-1}}]\eta
\]
for $1\leq t\leq m$ and
\[
q_tP\alpha=\sum^{2k}_{i=k+1}P_{it}p_i\alpha
\]
for  $m+1\leq t\leq 2m$.\qed
\end{proposition}

We proceed by determining the second order homotopy invariants of $P\circ\alpha$. Using decomposition~\eqref{decomP} for $P$, we have
\[
P\circ\alpha =(\sum^{2k}_{i=1}(\sum^{2m}_{s=1} j_sP_{is}p_i))\alpha=
\]
\[
\sum^{2k}_{i=1}\left( \sum^{2m}_{s=1} j_sP_{is}p_i\alpha+\sum_{1\leq s<t\leq 2m}[j_sP_{is}p_i, j_t P_{it}p_i]H(\alpha)\right)+ \sum_{1\leq i<j\leq 2k}\left( \sum^{2m}_{s=1}\sum^{2m}_{t=1}[j_sP_{is}p_i, j_tP_{jt}p_j]H(\alpha) \right)=
\]
\[
\sum^{2k}_{i=1}\sum^{2m}_{s=1}(j_sP_{is}p_i\alpha)+\sum^{2k}_{i=1}\sum_{1\leq s<t\leq 2m}[j_s, j_t]\Sigma (\bar P_{is}\wedge \bar P_{it})H_{ii}(\alpha) +\sum_{1\leq i<j\leq 2k} \sum^{2m}_{s=1}\sum^{2m}_{t=1}[j_s,j_t]\Sigma (\bar P_{is}\wedge \bar P_{jt})H_{ij}(\alpha).
\]

Notice that
\[
\Sigma (\bar \xi_{is}\wedge \bar \xi_{jt})H_{ij}(\alpha)\colon S^{2n-2}\lra\Sigma Y_s\wedge Y_t=\left\{\begin{array}{ll}
S^{2n-3} & \text{ for } 1\leq s,t\leq m\\
S^{2n-1}& \text{ for } m+1\leq s,t\leq 2m \\
S^{2n-2} & \text{ otherwise}\\
\end{array}\right.
\]
and that the corresponding homotopy groups $[ S^{2n-2}, \Sigma Y_s\wedge Y_t]$ are cyclic. We shall write $H_{ij}(\alpha)=h_{ij}^\alpha\eta^X_{ij}$ where $\eta^X_{ij}$ denotes a generator of $\pi_{2n-2}(\Sigma X_i\wedge X_j)$. Further, denote by $\eta_{st}$ a generator of $\pi_{2n-2}(\Sigma Y_s\wedge Y_t)$ and let
\[
\Sigma(\bar \xi_{is}\wedge\bar \xi_{jt})\eta^X_{ij}=\nu(i,j,s,t,n)\eta_{st}
\]
for some $\nu(i,j,s,t,n)\in\Z$. Using this notation, we have
\begin{equation}
\label{palpha}
P\circ\alpha=\sum^{2k}_{i=1}\sum^{2m}_{s=1}j_sP_{is}p_i\alpha+\sum^{m}_{s=1}\sum_{1\leq i<j\leq 2k} [j_s,j_s]a_{is}a_{js} \nu(i,j,s,s,n)h_{ij}^\alpha\eta_{ss}+
\end{equation}
\[
\sum_{1\leq s<t\leq 2m}[j_s,j_t]\left( \sum^{2k}_{i=1}a_{is}a_{it}h_{ii}^\alpha\nu(i,i,s,t,n)\eta_{st}+ \sum_{1\leq i<j\leq 2k }a_{is}a_{jt}h_{ij}^\alpha\nu(i,j,s,t,n)\eta_{st}+\right.
 \]

\[ \left.\sum_{1\leq i<j\leq 2k }a_{it}a_{js}h_{ij}^\alpha\nu(i,j,t,s,n)(-1)^{\dim \Sigma Y_s\dim \Sigma Y_t}\eta_{ts}
\right).
\]

By a straightforward calculations, one sees that $\nu(i,j,s,t,n)$ is trivial except in the following four type of cases when $\nu(i,j,s,t,n)=1$.

\vspace{5mm}
\begin{tabular}{llllc}  \toprule
$\Sigma X_i$ & $\Sigma X_j$ & $\Sigma Y_s$ & $\Sigma Y_t$ & $\nu(i,j,s,t,n)$\\ \midrule
$S^{n-1}$ & $S^{n-1}$ & $S^{n-1}$ &$S^{n-1}$ & 1\\
$S^{n-1}$ & $S^{n}$ & $S^{n-1}$ & $S^{n-1}$ & 1\\
$S^{n-1}$ & $S^{n}$ & $S^{n-1}$ &$S^{n}$ &1\\
$S^{n}$ &$S^{n-1}$ &$S^{n-1}$ &$S^{n-1}$ &1\\
$S^{n}$ &$S^{n-1}$ &$S^{n}$ &$S^{n-1}$ &1\\ \bottomrule
\hline
\end{tabular}
\vspace{3mm}

By Proposition~\ref{1storder} and the calculated values of $\nu(i,j,s,t,n)$,  the sum of the first two sums in~\eqref{palpha} equals to
\[
\sum^{2m}_{s=1} j_sq_sP\alpha.
\]

Using the facts that
\[
\eta_{st}=\left\{\begin{array}{ll}
\eta & \text{ for } 1\leq s,t\leq m\\
\Id_{S^{2n-2}} & \text{ for } 1\leq s\leq m, m+1\leq t\leq 2m \text{ and  for } 1\leq t\leq m, m+1\leq s\leq 2m\\
0 & \text{ for } m+1\leq s,t\leq 2m\\
\end{array}\right.
\]
and that
\[
h_{ij}^\alpha=\left\{\begin{array}{ll}
m_{ij}^\alpha & \text{ for } 1\leq i<j\leq k\\
1 & \text{ for } 1\leq i\leq k, j=k+i\\
0 &\text{ for } 1\leq i\leq k, k+1\leq j\leq 2k, j\neq k+1\quad\text{ and for } k+1\leq i,j\leq 2k\\
\end{array}\right.\]
together with the uniqueness of the decomposition in Proposition~\ref{HM},  we determine the second order homotopy invariants of $P\circ\alpha$.

Recall that $P_{is}=q_sP\iota_i\colon\Sigma X_i\lra \Sigma Y_s$ belongs to the cyclic group $[\Sigma X_i,\Sigma Y_s]$ generated by $\xi_{is}$ and that $P_{is}=a_{is}\xi_{is}$ for some $a_{is}\in\Z$. Analogously,$ H_{ij}(\alpha)\colon S^{2n-2}\lra\Sigma X_i\wedge X_j$ for $i\leq j$ belongs to a cyclic group and $h^\alpha_{ii}$, $m^\alpha_{ij}$ denote the multiplicity of the generator which gives $H_{ij}(\alpha)$.

\begin{proposition}
\label{2ndorder}
The second order homotopy invariants of $P\circ\alpha$ are given by

\[
M_{st}(P\circ \alpha)=\left\{\begin{array}{l}
\left( \sum^{k}_{i=1}a_{is}a_{it}h_{ii}^\alpha+\sum_{1\leq i<j\leq k}(a_{is}a_{jt}+a_{it}a_{js})m_{ij}^\alpha+\sum_{i=1}^{k}(a_{is}a_{i+k t}+a_{it}a_{i+k s}) \right)\eta\\
\hspace*{4.55cm}\text{ for } 1\leq s<t\leq m\\
(\sum^k_{i=1}a_{is}a_{i+k, t})\Id_{S^{2n-2}}\quad\hspace{.62cm} \text{ for } 1\leq s\leq m, m+1\leq t\leq 2m\\
0\quad\hspace{4cm}
\text{ for } m+1\leq s<t\leq 2m.\\
\end{array}\right.
\]\qed
\end{proposition}

\section{Degrees of maps}
In this section, we prove our main result which establishes necessary and sufficient algebraic conditions for existence of map degrees between torsion free $(n-2)$-connected $(2n-1)$-dimensional Poincar\' e complexes. To do so, we make a use of our definition of the degree of a map between spaces of the same dimension and with only one top dimensional cell, that is, between $\T^n$-spaces. Notice that $(n-2)$-connected $(2n-1)$-dimensional Poincar\'e complexes are examples of $\T^{2n-1}$-spaces.
\begin{lemma}
\label{invariantsbetad}
The invariants of $\beta [d]\colon S^{2n-2}\lra\bar Y=\Sigma \bigvee^m_{s=1}(S^{n-2}_s\vee S^{n-1}_{s+m}) $ are
\[
(\beta [d])_s=d\beta_s \quad\text{ for } 1\leq s\leq 2m
\]

\[
M_{st}(\beta [d])=dM_{st}(\beta) \quad\text{ for } 1\leq s<t\leq 2m.
\]
\end{lemma}
\begin{proof}
The proof follows from the fact that  $\beta [d]=d \beta$ and that all invariants are homomorphisms.
\end{proof}
\begin{theorem}
\label{degree}
Let $X,Y\in\J_n$ for $n>3$. Then there is a map $f\colon X\lra Y$ of degree $d$ if and only if the system of equations
\begin{equation}
\label {1.jednacina}
\sum^k_{i=1}a_{it}(p_i\alpha)+\sum^{2k}_{i=k+1}a_{it}(\eta p_i\alpha)+\left(\sum^k_{i=1}\binom{a_{it}}{2}h_{ii}^\alpha+\sum_{1\leq i<j\leq k}m^\alpha_{ij}a_{it}a_{jt}+\sum^k_{i=1}a_{it}a_{k+i t}\right)[\Id_{S^{n-1}},\Id_{S^{n-1}}]\eta\end{equation} \[=d(q_t\beta)\quad \text{ for }1\leq t\leq m .
\]
\begin{equation}
\label{2.jednacina}
\sum^{2k}_{i=k+1}a_{it}(p_i\alpha)=d(q_t\beta) \quad \text{ for } m+1\leq t\leq 2m 
\end{equation}
\begin{equation}
\sum^k_{i=1}a_{is}a_{it}h_{ii}^\alpha+\sum_{1\leq i<j\leq k}(a_{is}a_{jt}+a_{it}a_{js})m_{ij}^\alpha+\sum_{i=1}^k(a_{is}a_{i+k t}+a_{it}a_{i+k s})\equiv dm^\beta_{st}\quad   (\hspace{-3.5mm}\mod 2)
\end{equation}
for $1\leq s<t\leq m$ 
\begin{equation}
\label{4.jednacina}
\sum_{i=1}^k a_{is}a_{i+k t}=d\delta_{s, t-m} \quad\text{ for } 1\leq s\leq m, m+1\leq t\leq 2m
\end{equation}
has an integral solution $\{ a_{is} \ | \ 1\leq i\leq 2k, 1\leq s\leq 2m\}$.
\end{theorem}
\begin{proof}
By Lemma~\ref{lemmadegree}, if there is a map $f\colon X\lra Y$ of degree $d$, then the diagram

\begin{equation}
\label{diagramdegree}
\xymatrix{
S^{2n-2}\ar[r]^-{\alpha}\ar[d]^{[d]} & \bar X\ar[d]^{F} \\
S^{2n-2}\ar[r]^-{\beta} & \bar Y}
\end{equation}
commutes,
where $F$ is the restriction of $f$ on $\bar X$, that is, $\beta [d]=F\circ\alpha$.

Writing down the homotopy invariants of $F\circ\alpha$ and $\beta[d]$ as described in Propositions~\ref{1storder}, \ref{2ndorder} and Lemma~\ref{invariantsbetad}, respectively and comparing the appropriate terms, gives the required equations.

For the converse statement, recall that by~\eqref{mapP}, the map $F$ is determined uniquely by coefficients $a_{ij}$. Therefore if the system has a solution, then that solution defines a map $F$ such that diagram~\eqref{diagramdegree} commutes. Now the diagram of cofibrations
\[
\xymatrix{
S^{2n-2}\ar[r]^-{\alpha}\ar[d]^{[d]} & \bar X\ar[d]^{F}\ar[r] & X\ar[r]\ar@{-->}[d]^f & S^{2n-1}\ar[d]^{[d]}\\
S^{2n-2}\ar[r]^-{\beta} & \bar Y\ar[r] & Y\ar[r] & S^{2n-1}}
\]
defines a map $f$ of degree $d$.
\end{proof}

Note that we may in equations~\eqref{1.jednacina} use $ h(p_i\alpha )$ instead of $h^{\alpha }_{ii}$ by Proposition~\ref{Hij}~(3).

Note also that Theorem~\ref{degree} gives not only the existence of the map $f$ of degree $d$ , but also determines $f$ up to homotopy. The solution set $\{a_{ij}\}$ of the system determines $f|_{\bar X}$ and that combined with the degree $d$ map on the top cell defines the map $f$.

Moreover, if we could find all solutions of the system for a fixed $d$ that will decide the set of homotopy classes of all functions $f\colon X\lra Y$ of degree $d$.

Note that it was shown in~\cite{Mukai} that $[\Id_{S^{n-1}},\Id_{S^{n-1}}]\eta = 0$ for $n=7$ or $4\mid n$, so in that case equations~\eqref{1.jednacina} simplify.

A simple algebraic observation determines all possible degrees between certain Poincar\' e complexes.

Let $X,Y\in\J_n$ such that $X$ is of rank $k$ and $Y$ of rank $m$. For a map $P \colon \bar X \lra \bar Y $, using identity~\eqref{mapP} define $P^*$ to be the matrix $(a_{ij})_{2k \times 2m}$.

We state some properties of $P^*$ that we will need later.

\begin{lemma}
\label{vuckolema}
Let $X$,$Y$ and $Z$ be from $\J_n$, $P\colon \bar X \lra \bar Y$ and  $Q\colon \bar Y \lra \bar Z$. 
Let $\bar X=\Sigma \bigvee_{i=1}^k (S^{n-2}_i\vee S^{n-1}_{i+k})$, $\bar Y=\Sigma \bigvee^m_{s=1}(S^{n-2}_s\vee S^{n-1}_{s+m})$ and $\bar Z=\Sigma \bigvee^r_{t=1}(S^{n-2}_t\vee S^{n-1}_{t+r})$. Then:

\begin{itemize}
\item[(a)] $(QP)_{it}=\sum^m_{s=1}Q_{st}P_{is}$

\item[(b)] $((QP)^*)^T = (Q^*)^T (P^*)^T$

\item[(c)] If $k=m=r$, then  $det(QP)^* = det Q^* det P^* $.
\end{itemize}
\end{lemma}
\begin{proof}
By equation~\eqref{decomP}, there is
\[
QP=(\sum_{u,t}l_tQ_{ut}q_u)(\sum_{i,s}j_sP_{is}p_i)=
\]
\[
= \sum_{s,t,i}l_tQ_{st}P_{is}p_i= \sum_{t,i}l_t(\sum_sQ_{st}P_{is})p_i
\]
which proves (a). The rest of the proof  follows by simple linear algebra arguments.
\end{proof}

Define four corner matrices of $P^*$ of order $k \times m$ by
\[
A(P) := (a_{i\,s})
\]
\[
B(P):=(a_{i\,m+s}) = (0)
\]
\[
C(P) := (a_{k+i\,s})
\]
and
\[
D(P) := (a_{k+i\,m+s})
\]
for $1\leq i\leq k$ and $1\leq s\leq m$.

Then equations~\eqref{4.jednacina} can be rewritten as the matrix equation
\begin{equation}
\label{eqmatrix}
A(P)^TD(P)=dI_m.
\end{equation}

\begin{corollary}
\label{det}
Let $X,Y\in\J_n$ and let $f\colon X\lra Y$ be a map of degree $d$. If $k=\rank H_{n-1}(X)=\rank H_{n-1}(Y)=m$, then
\[
\det (f|_{\bar X})^* = d^k.
\]
\end{corollary}
\begin{proof}
\[
\det (f|_{\bar X})^* = \det(A(f|_{\bar X}))\det(D(f|_{\bar X}))=\det(dI_m)=d^k.
\]
\end{proof}

\begin{proposition}
\label{rank}
Let $X,Y\in\J_n$. If $k=\rank H_{n-1}(X)<m=\rank H_{n-1}(Y)$, then
\[
D(X,Y)=\{ 0\}.
\]
\end{proposition}

\begin{proof}
Recall that equations~\eqref{4.jednacina} of Theorem~\ref{degree} can be written as
\[
A(P)^TD(P)=d I_m.
\]
As $\rank A(P)\leq k<m$ and $\rank D(P)\leq k<m$, the matrix equation does not have a solution for $d\neq 0$.
\end{proof}

\section{Examples}
\label{examples}
In this section we explicitly calculate $D(X,Y)$ for maps between particular torsion free Poincar\' e complexes $X$ and $Y$. Start by requesting $X$ to be of rank~1. Let $\alpha\colon S^{2n-2}\lra S^{n-1}\vee S^n$ be the top cell attaching map of $X$ and let $\beta\colon S^{2n-1}\lra \Sigma \bigvee_{s=1}^m(S^{n-2}\vee S^{n-1})$ be the top cell attaching map of $Y$. Consider two cases depending on the rank of $Y$.

\vspace{3mm}
\noindent {\bf (a)}
If $m\geq 2$, then by Proposition~\ref{rank}, $D(X,Y)=\{0\}$.

\vspace{3mm}
\noindent {\bf (b)}
Let $m=1$. Then by Proposition~\ref{HM},

\[
\alpha=\iota_1p_1\alpha +\iota_2p_2\alpha +[\iota_1,\iota_2]
\]
\[
\beta=j_1q_1\beta+j_2q_2\beta+[j_1,j_2].
\]
The system of equations from Theorem~\ref{degree} becomes
\begin{equation}
\label{system1.1}
\begin{array}{l}
a_{11}(p_1\alpha)+a_{21}(\eta p_2\alpha)+(\binom{a_{11}}{2}h_{11}^\alpha +a_{11}a_{21})[\Id_{S^{n-1}},\Id_{S^{n-1}}]\eta=d(q_1\beta)\\
a_{22}(p_2\alpha)=d(q_2\beta)\\
a_{11}a_{22}=d.\\
\end{array}
\end{equation}

In general it is impossible to solve this system of equations, however for particular fixed values of $n,p_1\alpha,p_2\alpha, q_1\beta,q_2\beta$ one might be able to explicitly state a solution and therefore to determine precisely the set $D(X,Y)=\{ d\in\Z \ |\ \text{system~\eqref{system1.1} has a solution}\}$.

{\bf Case 1}: Let $n=4$, and $\alpha,\beta\colon S^6\lra S^{3}\vee S^4$ be arbitrary attaching maps. 
 The first order homotopy invariants of $\alpha\colon S^6\lra S^{3}\vee S^4$ are

 \[
 p_1\alpha\in\pi_6(S^3)\cong \Z/12=\langle w\rangle
 \]
 where $w=\nu'-\alpha_1(3)$ such that $\pi_6^3\cong\Z/4=\langle \nu'\rangle$, $2\nu'=\eta^3_3$ and
 \[
 p_2\alpha\in\pi_6(S^4)\cong \Z/2=\langle \eta^2_4\rangle \text{ where } \eta^2_4=\eta_4\eta_5.
 \]
Therefore if $X=(S^3\vee S^4)\cup_\alpha e^7$ is a Poincar\' e complex, then $\alpha$ can be written as
\[
\alpha=a(i_1w)+b(i_2\eta^2)+[i_1,i_2] \text{ for } a\in\Z/12, b\in\Z/2.
\]
Let $Y=(S^3\vee S^4)\cup_\beta e^7$ be another Poincar\' e complex where
\[
\beta=g(j_1w)+h(j_2\eta^2)+[i_1,i_2] \text{ for } g\in\Z/12, h\in\Z/2.
\]
As $S^3$ is an $H$-space, the Whitehead product $[\Id_{S^3},\Id_{S^3}]$ is trivial.
The system of equations given by Theorem~\ref{degree} becomes
\begin{equation}\begin{array}{l}
a_{11}aw+a_{21}b(\eta^3)=dgw\\
 a_{22}b\eta^2=dh \eta^2 \\
a_{11}a_{22}=d.\\
\end{array}
\end{equation}
Notice that $\eta^3$ has order 2. Therefore 
 \[
 \eta^3=6w.
\]
We get the system

\begin{equation}
\label{system2}
\begin{array}{l}
a_{11}a+6a_{21}b\equiv dg  \ (\hspace{-3.5mm}\mod 12)\\
a_{22}b\equiv dh\ (\hspace{-3.5mm}\mod 2)\\
a_{11}a_{22}=d.\\
\end{array}
\end{equation}
Therefore $D(X,Y)=\{ d\ | \ \text{system~\eqref{system2} has integral solutions } (a_{ij})\}$.

Let us consider some special cases.

{\bf Case 1(a)}: Let $g=h=0$.

Then $Y=S^3\times S^4$. System~\eqref{system2} becomes
\begin{equation}
\label{system4}
\begin{array}{l}
a_{11}a+6a_{21}b\equiv 0  \ (\hspace{-3.5mm}\mod 12)\\
a_{22}b\equiv 0\ (\hspace{-3.5mm}\mod 2)\\
a_{11}a_{22}=d.\\
\end{array}
\end{equation}
The solution of the system implies that

\[
\begin{array}{ll}
\text{if $b$ is odd, } &
 D(X,Y)=\{\frac{12}{(a,6)}k \ |\ k\in\Z \} \\
\text{if $b$ is even, } & D(X,Y)=\{\frac{12}{(a,12)}k\ | \ k\in\Z\} .\\
\end{array}
\]

{\bf Case 1(b)}: Let $a=b=g=h=1$.

Then $X=Y$.
System~\eqref{system2} becomes
\begin{equation}
\label{system5}
\begin{array}{l}
a_{11}+6a_{21}\equiv d  \ (\hspace{-3.5mm}\mod 12)\\
a_{22}\equiv d\ (\hspace{-3.5mm}\mod 2)\\
a_{11}a_{22}=d.\\
\end{array}
\end{equation}
The solution of the system implies that
\[
D(X,X)=\{ d\in \Z\ \mid \ d\not\equiv 2\ (\hspace{-3.5mm}\mod 4)\}.
\]

{\bf Case 2}:
In what follows, we should allow for the rank of $X$ to be arbitrary, but restrict a choice of the attaching map and keep the rank of $Y$ to be 1.

\begin{proposition}
Let $X=(S^{n-1}\times S^n)^{\# k}$ and $Y=(S^{n-1}\vee S^n)\cup_{\beta}e^{2n-1}$ be from $\J_n$, for $n>3$.

(a) If $n=7$ or $4|n$ or if $n\neq 7$ and $4\nmid n$ and there is no $t\in \N$ such that $t q_1\beta = [\Id_{S^{n-1}},\Id_{S^{n-1}}]\eta$, then
\[
D(X,Y)=\{\lcm(\mathrm{order}(q_1\beta ),\mathrm{order}(q_2\beta ))k \ | \ k\in\Z\}.
\]

(b) If $n\neq 7$, $4\nmid n$ and there is $t\in \N$ such that $t q_1\beta = [\Id_{S^{n-1}},\Id_{S^{n-1}}]\eta$, then
\[
D(X,Y)=\{\lcm(\mathrm{order}(q_1\beta )/2,\mathrm{order}(q_2\beta ))k\ |\ k\in\Z\}.
\]
\end{proposition}
\begin{proof}
(a) If $n=7$ or $4|n$ then $[\Id_{S^{n-1}},\Id_{S^{n-1}}]\eta=0$ (see~\cite{Mukai}), so the equations of Theorem~\ref{degree} become
\[
dq_1\beta =0
\]
\[
dq_2\beta = 0
\]
\[
\sum_{i=1}^k a_{i1}a_{i+k 2}=d.
\]
Hence, necessary conditions for the existence of a map $f\colon X \lra Y$ of degree $d$ are $\mathrm{order}(q_1\beta )|d$ and $\mathrm{order}(q_2\beta )|d$. Note that the order of the trivial map is 1. These conditions are sufficient as well, since for such $d$ we may choose $a_{11}=d$, $a_{k+1 2}=1$ and $a_{ij}=0$ in other cases. This proves (a) for  $n=7$ or $4|n$.

If $n\neq 7$ and $4\nmid n$ and there is no $t\in \N$ such that $t q_1\beta = [\Id_{S^{n-1}},\Id_{S^{n-1}}]\eta$, then both sides of equation~\eqref{1.jednacina} in Theorem~\ref{degree} have to be equal to 0. Using Theorem~\ref{degree}, we have
\[
\sum_{i=1}^ka_{i1}a_{k+i 1} \equiv 0\ (\hspace{-3.5mm}\mod 2)
\]
\[
dq_1\beta =0
\]
\[
dq_2\beta = 0
\]
\[
\sum_{i=1}^k a_{i1}a_{i+k 2}=d.
\]
Also in this case, necessary conditions for the existence of a map $f\colon X \lra Y$ of degree $d$ are that $\mathrm{order}(q_1\beta )|d$ and $\mathrm{order}(q_2\beta )|d$. For such $d$, choose $a_{11}=d$, $a_{k+1 2}=1$ $a_{k+1 1} =2$ and $a_{ij}=0$, otherwise. This proves that a map $f\colon X \lra Y$ of degree $d$ exists and finishes the proof of part (a).

(b) Assume that $n\neq 7$ and $4\nmid n$, so that $[\Id_{S^{n-1}},\Id_{S^{n-1}}]\eta \neq 0$. Let $t_0$ be a minimal $t\in\mathbb{N}$ such that $t(q_1\beta)=[\Id_{S^{n-1}},\Id_{S^{n-1}}]\eta$. Note that $\mathrm{order}(q_1\beta)=2t_0$.

The equations of Theorem~\ref{degree} become
\[
(\sum_{i=1}^ka_{i1}a_{k+i 1})[\Id_{S^{n-1}},\Id_{S^{n-1}}]\eta = dq_1\beta
\]
\[
dq_2\beta = 0
\]
\[
\sum_{i=1}^k a_{i1}a_{i+k 2}=d.
\]

Then, for the above system to have a solution, necessary conditions are
\[
\begin{array}{lll}
\mathrm{order}(q_2\beta)\ |\ d \\
t_0\ |\ d. & &\\
\end{array}
\]

Therefore, $d\in \{\lcm(\mathrm{order}(q_2\beta), \mathrm{order}(q_1\beta)/2)k \ | \ k\in\Z\}$.

Choose such $d$. 
If $d$ is an odd multiple of $t_0$. Let $2^s$ be a maximal power of 2 which divides $d$ and choose $a_{11}=d/2^s$, $a_{k+1 2}=2^s$, $a_{k+1 1}=1$ and $a_{ij}=0$, otherwise, and it proves that there exists a map $f\colon X \lra Y$ of degree $d$.

If $d$ is an even multiple of $t_0$, choose $a_{11}=d$, $a_{k+1 1}=a_{k+1 2}=1$ and $a_{ij}=0$ otherwise. All equations are satisfied, so there is a map $f$ of degree $d$, which completes the proof.
\end{proof}

\begin{proposition}
Let $X=(S^{n-1} \times S^n)^{\# k}$ and $Y=(S^{n-1} \times S^n)^{\# m}$ for $n>3$. Then
\begin{itemize}
\item[(a)] $D(X,Y) = \{0\}$  for $k<m$

\item[(b)] $D(X,Y) = \Z$\hspace{.3cm}  for $ k\geq m$.
\end{itemize}
\end{proposition}
\begin{proof}
(a) follows from Proposition~\ref{rank}.

(b) Let $ k\geq m$. In this case the equations of Theorem~\ref{degree} become
\[
(\sum_{i=1}^ka_{it}a_{k+i\, t})[\Id_{S^{n-1}},\Id_{S^{n-1}}]\eta =0
\]
\[
\sum_{i=1}^k(a_{is}a_{i+k\, t} + a_{it}a_{i+k\, s}) = 0
\]
\[
\sum_{i=1}^ka_{is}a_{i+k\, t} = d\delta_{s,t-m}.
\]
 
 For each $d\in \Z$, the system have a solution. For example, 
 \[
 a_{11}=a_{22}=...=a_{mm}=d
 \]
 \[
 a_{k+1\, m+1}=a_{k+2\, m+2}= ... = a_{k+m\, 2m}=1
 \]
 and $a_{ij}=0$ otherwise.
\end{proof}

\section{Homotopy classification of Poincar\' e complexes}
The methods that we have developed might be used to determine the number of different homotopy types of Poincar\' e complexes belonging to $\J_n$ with a fixed rank $k$. We state preliminary results.

\begin {theorem}
\label{heq1}
For $X,Y \in \J_n$ with the same rank and a map $f\colon X\lra Y$, the following statements are equivalent.
\begin{itemize}
\item [(a)] A map $f$ is a homotopy equivalence. 
\item [(b)] The restriction $f|_{\bar X}\colon \bar X\lra \bar Y$ is a homotopy equivalence.
\item [(c)] The degree $\deg (f)=\pm 1$.
\end{itemize}
\end {theorem}
\begin{proof}

\noindent $(a)\Longrightarrow (b)$ As all maps are cellular, a restriction of a homotopy equivalence is a homotopy equivalence.

\noindent $(b)\Longrightarrow (c)$ Assume that $P=f|_{\bar X}$ is a homotopy equivalence, whose inverse is $Q\colon \bar Y \lra \bar X $ and let $d=\deg(f)$. Since $QP=\Id_{\bar X}$, by Lemma~\ref{vuckolema} and Corollary~\ref{det}
\[
1 = \det (QP)^* = \det (Q^*)^T \det (P^*)^T = \det Q^* \cdot d^k.
\]
Therefore $d=\pm 1$.

\noindent $(b)\Longrightarrow (a)$ In the previous step, we proved that $f|_{\bar X}$ being a homotopy equivalence implies that $\deg(f)=\pm 1$. Therefore,
\[
Q\beta = QP\alpha [d]=\alpha [d].
\]
By Lemma~\ref{lemmadegree}, there is a map $ h\colon Y \lra X $ of degree $d$ extending $Q$ and it is a homotopy inverse of $f$.

\noindent $(c)\Longrightarrow(b)$Assuming that $\deg(f)=\pm 1$, by  equation~\eqref{eqmatrix}, $\det A(f|_{\bar X})=\pm 1$. Therefore $A(f|_{\bar X})$ is invertible and $D(f|_{\bar X})=d(f|_{\bar X})^{-1}$. Define $Q\colon \bar Y\lra\bar X$ by
\[
\begin{array}{l}
A(Q)=A(f|_{\bar X} )^{-1}\\
D(Q)=\pm A(f|_{\bar X} )^T\\
C(Q)=A(f|_{\bar X} )^TC(f|_{\bar X} )A(f|_{\bar X} )^{-1}.\\
\end{array}
\]
It follows readily that $Q$ is an inverse of $f|_{\bar X}$.
Hence $f|_{\bar X}$ is a homotopy equivalence.
\end {proof}

In particular, for $k=1$ situation simplifies and we get the following.

\begin {corollary}
\label{posledica}
Let $X,Y \in J_n$ be of rank 1, and let $f \colon X \lra Y$ be a map between them.
Then
\begin{itemize}
\item [(a)] A map $f$ is a homotopy equivalence if and only if
\[
f|_{\bar X} = a_{11}j_1p_1 + a_{21}j_1\eta p_2 + a_{22}j_2p_2
\]
where $a_{11}=\pm 1$, $a_{22}=\pm 1$ and $a_{21} \in \{ 0,1 \}$.

\item[(b)] If $f|_{\bar X}\colon S^{n-1}\vee S^n \lra S^{n-1}\vee S^n$ is a homotopy equivalence, then $f|_{\bar X}$ is its own inverse, that is, $f|_{\bar X}\circ f|_{\bar X} = \Id_{S^{n-1}\vee S^n}$.
\end{itemize}\qed
\end {corollary}

\begin{theorem}
There are 11 different homotopy types of Poincar\' e complexes in $\J_4$ of rank 1.
\end{theorem}
\begin{proof}
In order to have a homotopy equivalence $f\colon X \lra Y $ for $X,Y\in \J_4$, Corollary~\ref{posledica} implies that the second equation of system~\eqref{system2} reduces to $b=h$. Recall that in this case $d=\pm 1$, $a_{11}=a_{22}=\pm1$. There are two possible values for $b$.

If $b=h = 0$, system~\eqref{system2} reduces to

\[
a\equiv a_{22}g  \ (\hspace{-3.5mm}\mod 12).
\]
If $a_{22}=1$, we get $X=Y$.
If $a_{22}=-1$, we get $a\equiv -g\ (\hspace{-2mm}\mod 12)$.
Therefore, we have 7 different homotopy types of Poincar\' e complexes in $\J_4$ of rank 1 with $b=0$, which are specified by $a=0$, $a \in \{ 1,11\}$, $a\in\{2,10\}$, $a\in\{3,9\}$, $a\in\{4,8\}$, $a\in\{5,7\}$ and $a=6$.

If $b=h = 1$, system~\eqref{system2} reduces to

\[
a_{11}a+6a_{21}\equiv a_{11}a_{22}g  \ (\hspace{-3.5mm}\mod 12).
\]

If $a_{21}=0$, then $ a\equiv \pm g \ (\hspace{-2mm}\mod 12)$.

If $a_{21}=1$, then $ a+6\equiv \pm g  \ (\hspace{-2mm}\mod 12)$.

Therefore, we have 4 different homotopy types of Poincar\' e complexes in $\J_4$ of rank 1 with $b=1$, determined by $a\in \{0,6\}$, $a \in \{ 1,5,7,11\}$, $a\in\{2,4,8,10\}$ and $a\in\{3,9\}$.

To summarise, there are 11 different homotopy types of Poincar\' e complexes in $\J_4$ of rank 1.
\end{proof}

\begin{theorem}
There are 38 different homotopy types of Poincar\'e complexes in $\J_5$ of rank 1.
\end{theorem}
\begin{proof}
For the case $n=5$, the following holds (see \cite{Toda} or \cite{Mukai} ).
The homotopy groups we need are $ \pi_8(S^4)=\Z/2 \oplus \Z/2 $ with generators  $\nu_4\eta_7$ and $E\nu '\eta_7$, which we denote by $\epsilon_1,\epsilon_2$, respectively, and $\pi_8(S^5)=\Z/24$ with a generator $\nu_5+\alpha_1(5)$, which we denote by $w$, where $\nu_5$ is of order 8 and $\alpha_1(5)$ is of order 3.

The following relations hold.
\[
[\Id_4,\Id_4]\eta=\epsilon_2, \eta w=\epsilon_2 , H(\epsilon_1)=\eta , H(\epsilon_2)=0, h_{ii}(a\epsilon_1+b\epsilon_2)=a.
\]

For $k=m=1$ and $n=5$, let
\[
\alpha = i_1(a\epsilon_1+b\epsilon_2)+i_2cw+[i_1,i_2]
\]
and
\[
\beta = j_1(A\epsilon_1+B\epsilon_2)+j_2Cw+[j_1,j_2].
\]
be the attaching maps of $X$ and $Y$, respectively.

The equations of Theorem~\ref{degree} become
\[
a_{11}a\equiv dA  \ (\hspace{-3.5mm}\mod 2)
\]
\[
a_{11}b+a_{21}c + \binom{a_{11}}{2}a+a_{11}a_{21} \equiv dB  \ (\hspace{-3.5mm}\mod 2)
\]
\[
a_{22}c \equiv dC  \ (\hspace{-3.5mm}\mod 24)
\]
\[
a_{11}a_{22} = d.
\]

Since we are looking for a homotopy equivalence between $X$ and $Y$, by Theorem~\ref{heq1} the degree $d=\pm 1$. Therefore, the last equation implies that $a_{11}=\pm 1$ and $a_{22}=\pm 1$. The first equation implies that $a=A$.

The system reduced to the following two equations
\[
b+a_{21}c + \binom{a_{11}}{2}a+a_{21} \equiv B  \ (\hspace{-3.5mm}\mod 2)
\]
\[
c \equiv a_{11}C  \ (\hspace{-3.5mm}\mod 24).
\]

If $c$ is even, then  $a_{21}$ can be whatever it has to be to satisfy the first equation and we are left with
\[
c \equiv \pm C  \ (\hspace{-3.5mm}\mod 24).
\]

There are 7 various values for $c$ even, which are not mutually equivalent, with 2 possible values for $a$ in each case. This gives that for $c$ even, there are 14 different homotopy types of Poincar\' e complexes in $\J_5$ of rank 1.

For $c$ odd and $a_{11}=1$, the system of equations become $b=B$ and $c=C$. Therefore, $\alpha=\beta$ giving a homotopy type which has been already counted.

For $c$ odd and $a_{11}=-1$, the system of equations become
\[
b+a=B
\]
\[
c=-C.
\]

Note that the equations uniquely determine $\beta$ and since $c$ is odd, $c=-C$ guarantees that $\alpha \neq \beta$.

Therefore,  for $c$ odd, there is exactly one other value for $\beta$, so that $\alpha$ and $\beta$ give homotopy equivalent Poincar\' e complexes.

Hence, for $c$ odd, there are exactly 24 different homotopy types of Poincar\'e complexes in $\J_5$ of rank 1.
\end{proof}

\begin{theorem}
Every Poincar\' e complex in $\J_6$ of rank $k$ is homotopy equivalent to $(S^5 \times S^6)^{\# k}$.
\end{theorem}
\begin{proof}
In this case we have $\pi_{10}(S^5)=\Z /2 $ generated by $\nu_5\eta^2_8$ , $\pi_{10}(S^6)=0$, and $[\Id_{S^5},\Id_{S^5}]\eta=\nu\eta^2$.

Note that $\alpha\colon S^{10}\lra \bigvee_{l=1}^k(S^5_l\vee S^6_l)$ given by
\[
\alpha = \sum_{l=1}^k [i_l,i_{k+l}]
\]
is the top cell attaching map of $X=(S^5 \times S^6)^{\# k}$. Let
\[
\beta = \sum_{t=1}^kj_tq_t\beta + \sum_{1\leq s < t \leq k} m_{st}[j_s,j_t]\eta +\sum_{t=1}^k [j_t,j_{k+t}]
\]
 be a top cell attaching map of some other Poincar\'e complex $Y$ from $\J_6$ of rank $k$.

Theorem~\ref{heq1} implies that there is a homotopy equivalence $h\colon X\lra Y$ since system~\eqref{system1.1} has a solution $(a_{ij})$ when assume that the degree $d=1$. The homotopy equivalence $h$ is determined by
\[
\begin{array}{ll}
a_{ij} = \delta_{ij} & \text{ for } 1\leq i,j \leq k, \text{ or } k+1 \leq i,j \leq 2k\\
a_{k+t t} \nu_5\eta^2_8= q_t\beta & \text{ for }  1 \leq t \leq k\\
a_{ij} = m_{i-k j} & \text{ for } i-j <k, 1\leq j \leq k, k+1 \leq i \leq 2k\\
a_{ij}=0 & \text{ otherwise. }\\
\end{array}
\]

Therefore, there is only one homotopy type of Poincar\' e complexes in $\J_6$ of rank $k$, namely $(S^5 \times S^6)^{\# k}$.

Note that
\[
D((S^5 \times S^6)^{\# k},(S^5 \times S^6)^{\# k})=\Z
\]
as for each $d \in \Z$ we have a map of degree $d$ defined by
\[
\begin{array}{ll}
a_{ij}=d & \text{ for }  1\leq i=j\leq k\\
a_{ij}=1 & \text{ for }  k+1 \leq i=j \leq 2k\\
a_{ij}=0 & \text{ otherwise.}\\
\end{array}
\]
\end{proof}

Now, consider the case $n=7$.

In this case, $\pi_{12}(S^6)=\Z /2$ generated by $\nu_6^2=\nu_6\nu_9$, $\pi_{12}(S^7)=0$, ${[\Id_{S^6},\Id_{S^6}]\eta=0}$ and  $H(\nu^2)=h_{11}\eta=0$ (see in \cite[Proposition 5.11]{Toda} and \cite{Mukai}).

For $k=1$, there are only two homotopy classes of the top cell attaching maps  and are represented by $ [i_1,i_2]$ and $i_1\nu^2 +[i_1,i_2]$. Denote the corresponding Poincar\'e complexes by $W_1 = S^6 \times S^7 $ and $Z_1$.

\begin{definition}
Let  $X$ and $Y$ be $n$-dimensional $CW$-complexes with one top cell attached by $\alpha,\beta$, respectively. Denote by $\bar X=\sk_{n-1}(X)$. Define a {\it homotopy connected sum of $X$ and $Y$}, denoted by $X\# Y$, as the homotopy cofibre of the map
\[
\alpha +\beta\colon S^{n-1}\lra \bar X\vee\bar Y.
\]
\end{definition}
Note that this definition is a natural homotopy theoretical generalisation of the classical connected sum operation between manifolds.
 
Define $W_k$ and $Z_k$ as Poincar\' e complexes in $\J_7$ of rank $k$ by
\[
W_k= W_1^{\#k}
\]
and
\[
Z_k = Z_1\#W_1^{\#(k-1)}.
\]

\begin{theorem}
\label{n=7}
For $k\geq 1$, there are two different homotopy types of Poincar\' e complexes in $\J_7$ of rank $k$ given by $W_k$ and $Z_k$.\qed
\end{theorem}

Before proving Theorem~\ref{n=7}, some preliminary observations are needed.

\begin{lemma}
Let
\[
\alpha =\sum_{i=1}^ki_ip_i\alpha + \sum_{1\leq i < j \leq k} m_{ij}^{\alpha}[i_i,i_j]\eta + \sum_{i=1}^k [i_i,i_{k+i}]
\]
and
\[
\beta = \sum_{t=1}^kj_tq_t\beta + \sum_{1\leq s < t \leq k} m_{st}^{\beta}[j_s,j_t]\eta +\sum_{t=1}^k [j_t,j_{k+t}]
\]
be top cell attaching maps of some Poincar\'e complexes $X$ and $Y$ in $\J_7$ of rank $k$.
If $|\{i|p_i\alpha \neq 0\}| = |\{s|q_s\beta \neq 0\}|$, then $X \simeq Y$ and $D(X,Y)=\Z$.
\end{lemma}

\begin{proof}
Let $r=|\{i|p_i\alpha \neq 0\}| \geq 0$. By permuting the spheres in $\bar X$ and $\bar Y$, assume that
\[
p_1\alpha=...=p_r\alpha=q_1\beta=...=q_r\beta=\nu^2
\]
and
\[
p_{r+1}\alpha=...=p_k\alpha=q_{r+1}\beta=...=q_k\beta=0.
\]
In order to keep the cup product matrix to be the identity matrix, permute in the same way$(n-1)$-spheres and $n$-spheres.

Suppose $m_{ij}^{\alpha}=0$ for all $1\leq i < j \leq k$ and denote the obtained space by $X_0$. Define $P\colon \bar X_0 \lra \bar Y$ by
\[
a_{ij}=\left\{\begin{array}{ll}
d & \text{    for } 1\leq i=j\leq k\\
1& \text{    for } k+1 \leq i=j \leq 2k\\
m_{st}^{\beta} & \text { for } i=k+s,\text{ }j=t, \text { } 1 \leq s < t \leq k\\
0&  \text{    otherwise}.\\
\end{array}\right.
\]

Since $P$ satisfies all equations of Theorem~\ref{degree} we conclude that $P$ extends to a map of degree $d$. Therefore $D(X_0,Y)=\Z$.

Fix $d=1$. Then by Theorem~\ref{heq1}, $P$ is a homotopy equivalence and therefore$X_0\simeq Y$. This completes the proof of the lemma as $X\simeq X_0 \simeq Y$.
\end{proof}

\begin{lemma} With the notation as above.

(a) If $q_s\beta =0$ for all $1\leq s \leq k$, then $Y \simeq W_k$.

(b) If $q_s\beta \neq 0$ for at least one $1\leq s \leq k$, then $Y \simeq Z_k$.
\end{lemma}
\begin{proof}
Part (a) follows directly from the previous lemma when $|\{i|p_i\alpha \neq 0\}|=0$.

For part (b), assume $X=Z_k$ and $r=|\{s|q_s\beta \neq 0\}|>1$. Therefore
\[
\alpha = i_1\nu ^2 + \sum_{i=1}^k [i_i,i_{k+i}]
\]
and assume, using the previous lemma, that all $m^\beta_{st}=0$, that is,
\[
\beta = \sum_{s=1}^r j_s\nu ^2 + \sum_{s=1}^k [j_s,j_{k+s}].
\]

For $d=1$, define $P\colon \bar X \lra \bar Y$ by
\[
a_{ij}=\left\{\begin{array}{ll}
1& \text{   for } i=j \text { or } i=1,1\leq j \leq r\\
-1& \text{  for } k+2\leq i \leq k+r, j=k+1\\
0& \text{   otherwise.}\\
\end{array}\right.
\]

The map $P$ satisfies all equations of Theorem~\ref{degree} for $d=1$, and by Theorem~\ref{heq1} it induces a homotopy equivalence between $Z_k$ and $Y$.
\end{proof}

\begin{corollary}
\[
D(W_k,W_k) = D(Z_k,Z_k) = \Z
\]\qed
\end{corollary}
\begin{proposition}
\[
D(W_k,Z_k) = D(Z_k, W_k) = 2\Z.
\]
\end{proposition}
\begin{proof}
For maps $W_k\lra Z_k$, equation~\eqref{1.jednacina} of Theorem~\ref{degree} becomes
\[
d\nu ^2 =0.
\]
Therefore, there are no maps of odd degree since $\nu ^2$ is of order 2.

For $d$ even, consider map $P\colon \bar W_k \lra \bar Z_k$ defined by
\[
a_{ij}=\left\{\begin{array}{ll}
d& \text{   for } 1\leq i=j  \leq k\\
1& \text{   for } k+1\leq i= j \leq 2k\\
0& \text{   otherwise}.\\
\end{array}\right.
\]
It satisfies all equations of Theorem~\ref{degree} and induces a map of degree $d$.

Regarding maps between $Z_k$ and $W_k$, equations~\eqref{2.jednacina} give that for a map $P\colon \bar Z_k \lra \bar W_k$ coefficients $a_{11}, a_{12},...,a_{1k}$ have to be even. Then $\det A(P)$ is even and because
\[
d^k=\det P^*=\det A(P)\det D(P)
\]
degree $d$ has to be even.

For $d$ even, consider the map $P\colon \bar Z_k \lra \bar W_k$ defined by
\[
a_{ij}=\left\{\begin{array}{ll}
d& \text{   for } 1\leq i=j  \leq k\\
1& \text{   for } k+1\leq i= j \leq 2k\\
0& \text{   otherwise}.\\
\end{array}\right.
\]
It satisfies all equations of Theorem~\ref{degree} and extends to a map between $Z_k$ and $W_k$ of degree $d$.
\end{proof}
\begin{proof}[Proof of Theorem~\ref{n=7}]
The above analysis determines all different homotopy types of Poincar\' e complexes in $\J_7$ of rank $k$.
\end{proof}

\end{document}